\documentclass[12pt,a4paper]{article}
\usepackage{}

\usepackage{latexsym}
\usepackage{amsmath}
\usepackage{amssymb}
\usepackage{arydshln}

\usepackage{cite}
\usepackage{stmaryrd}
\usepackage{enumerate}

\usepackage{hyperref}

\usepackage{color}
\usepackage{lineno}
\usepackage{graphicx}
\usepackage{ae}
\usepackage{amsmath}
\usepackage{amssymb}
\usepackage{latexsym}
\usepackage{url}
\usepackage{epsfig}
\usepackage{mathrsfs}
\usepackage{amsfonts}
\usepackage{amsthm}

\usepackage{float}
\usepackage{subfig}
\usepackage{tikz}
\usetikzlibrary{positioning}

\newtheorem{theorem}{Theorem}[section]

\newtheorem{lemma}[theorem]{Lemma}

\theoremstyle{definition}

\numberwithin{equation}{section} 
\allowdisplaybreaks    

\def\qed{\hfill$\Box$\vspace{12pt}}

\long\def\delete#1{}

\usepackage{xcolor}
\usepackage[normalem]{ulem}

\begin{document}
\title {Laplacian pretty good edge state transfer in paths}

\author{Wei Wang$^{a,b}$,~Xiaogang Liu$^{a,b,}$\thanks{Supported by the National Natural Science Foundation of China (No. 11601431), the Natural Science Foundation of Shaanxi Province (No. 2020JM-099) and the Natural Science Foundation of Qinghai Province  (No. 2020-ZJ-920).}~$^,$\thanks{ Corresponding author. Email addresses: wang-wei@mail.nwpu.edu.cn, xiaogliu@nwpu.edu.cn, wj66@mail.nwpu.edu.cn}~,~Jing Wang$^{a,b}$
\\[2mm]
{\small $^a$School of Mathematics and Statistics,}\\[-0.8ex]
{\small Northwestern Polytechnical University, Xi'an, Shaanxi 710072, P.R.~China}\\
{\small $^b$Xi'an-Budapest Joint Research Center for Combinatorics,}\\[-0.8ex]
{\small Northwestern Polytechnical University, Xi'an, Shaanxi 710129, P.R. China}
}

\date{}

\openup 0.5\jot
\maketitle

\begin{abstract}

In this paper, we first give a necessary and sufficient condition for a graph to
have Laplacian pretty good pair state transfer. As an application of such result, we give a complete characterization of Laplacian pretty good edge state transfer in paths.

\smallskip


\end{abstract}

\section{Introduction}

Let $G$ be a graph and let $a$ and $b$ be two distinct vertices in $G$. We usually use $\{a,b\}$ to denote such a pair of vertices in $G$. Let $\mathbf{e}_{a}$ denote the unit vector with only the $a$-th entry equal to $1$ and all other entries equal to $0$. Then $\mathbf{e}_a-\mathbf{e}_b$ denotes the \emph{pair state} of $\{a,b\}$. In particular, if $a$ and $b$ are connected by an edge in $G$, then such a pair state is called an \emph{edge state}. Let $A_{G}$ be the adjacency matrix of $G$, and let $L_G=D_G-A_G$ denote the Laplacian matrix of $G$, where $D_G$ is the diagonal matrix with diagonal entries the degrees of vertices of $G$. The \emph{transition matrix} of $G$ relative to $L_{G}$ is defined by
$$U_{L_{G}}(t) = \exp(-\mathrm{i}tL_{G})=\sum_{k\geq0}\frac{(-\mathrm{i})^{k}L_{G}^{k}t^{k}}{k!}, ~ t \in \mathbb{R},~\mathrm{i}=\sqrt{-1}.$$
Let $\{a,b\}$ and $\{c,d\}$ be two pairs of vertices in $G$. We say $G$ has \emph{Laplacian perfect pair state transfer}
(Pair-LPST for short) \cite{QCh18, Chen20} from $\{a,b\}$ to $\{c,d\}$ at time $\tau$ if and only if
$$U_{L_{G}}(\tau)(\mathbf{e}_{a}-\mathbf{e}_{b})=\gamma(\mathbf{e}_{c}-\mathbf{e}_{d})$$
where $\gamma$ is a complex scalar and $\left|\gamma\right|=1$, which is equivalent to
\begin{align*}
\left|\frac{1}{2}(\mathbf e_{a}-\mathbf e_{b})^\top e^{-\mathrm{i}\tau L_{G}}(\mathbf e_{c}-\mathbf e_{d})\right|^{2}=1,
\end{align*}
where $\ast^\top$ denotes the transpose of $\ast$. In this case,  we also say that $\mathbf e_{a}-\mathbf e_{b}$ (respectively, $\mathbf e_{c}-\mathbf e_{d}$)  has  Pair-LPST; otherwise, $\mathbf e_{a}-\mathbf e_{b}$ (respectively, $\mathbf e_{c}-\mathbf e_{d}$)  has  no Pair-LPST. Note that Pair-LPST does not always exist in a graph. Thus, a relaxation of Pair-LPST, the \emph{Laplacian pretty good pair state transfer} (short for Pair-LPGST), was proposed \cite{WangLJ21}. We say a graph $G$ has Pair-LPGST from $\{a,b\}$ to $\{c,d\}$ if for any $\epsilon>0$, there exists a time $\tau$ such that
$$
\left|U_{L_{G}}(\tau)(\mathbf{e}_{a}-\mathbf{e}_{b}) -\gamma(\mathbf{e}_{c}-\mathbf{e}_{d})\right|<\epsilon,
$$
where $\gamma$ is a complex scalar and $\left|\gamma\right|=1$, which is equivalent to
\begin{align*}
\left|\frac{1}{2}(\mathbf e_{a}-\mathbf e_{b})^\top e^{-\mathrm{i}\tau L_{G}}(\mathbf e_{c}-\mathbf e_{d})\right|^{2}> 1-\epsilon.
\end{align*}
If both $\mathbf e_{a}-\mathbf e_{b}$ and $\mathbf e_{c}-\mathbf e_{d}$ are edge states in above definitions, then Pair-LPST and Pair-LPGST are also called Edge-LPST and Edge-LPGST, respectively.

Pair state transfers can be regarded as generalizations of state transfers, which are very important in quantum computing and quantum information processing \cite{Chri17,Kay10,Kemp17}. Up until now, many graphs have been shown to admit state transfers, including trees \cite{Ban17, Bose03, CoutinhoL2015, CouG17, GOD12}, Cayley graphs \cite{Basic13, CaoCL20, CaoF21, CaoWF20, CC, HPal3, HPal5, Tan19}, distance regular graphs \cite{Coutinho15} and some graph operations such as NEPS \cite{chris1, chris2, LiLZZ21, HPal1, HPal4, SZ}, coronas \cite{Ack, AckBCMT16} and joins \cite{Angeles10, AlvirDLM16}. For more information, we refer the reader to \cite{Coutinho14, Coh19, Eisen19, CGodsil, Godsil12, HZ, Zhou14}. Numerical results in \cite{Chen20} have shown that there should exist more graphs having pair state transfers than state transfers for a certain number of vertices. However, only few results on pair state transfers are given now. In 2018, Chen \cite{QCh18}  first proposed the concept of Edge-LPST in his thesis. He presented many results on Edge-LPST, including the existence of Edge-LPST in paths and cycles, and he also gave two methods to construct new graphs with Edge-LPST. Those results can also be found in the published Paper \cite{Chen20}, where Pair-LPST was used instead of Edge-LPST. In 2021, Luo, Cao, et al. \cite {LCXC21} gave necessary and sufficient conditions for the existence of Edge-LPST in Cayley graphs of dihedral groups. With these conditions, the authors also proposed several concrete constructions of Cayley graphs admitting Edge-LPST. In the same year, Cao \cite{Cao2021} gave a characterization of cubelike graphs having Edge-LPST. The author also gave some concrete constructions which can be used to obtain some classes of infinite graphs possessing Edge-LPST. In 2022, Wang, Liu, et al. \cite{WangLJ21} gave sufficient conditions for vertex corona to have or to not have Pair-LPST. They also give a sufficient condition for vertex corona to have Pair-LPGST. These seem to be the newest results on Pair-LPST and Pair-LPGST.

As mentioned above, only few results on Pair-LPST and Pair-LPGST are given now. Thus, in this paper, we continue to study this issue. In this paper, we first give a necessary and sufficient condition for a graph to
have Pair-LPGST. As an application of such result, we give a complete characterization of Edge-LPGST in paths.

\section{Necessary and sufficient condition for a graph to have Pair-LPGST}
\label{Sec:main2}

In this section, we give a necessary and sufficient condition for a graph to have Pair-LPGST. Before proceeding, we present some notations and results.

Let $G$ be a graph on $n$ vertices with its Laplacian matrix $L_G$.  The eigenvalues of $L_G$ are called the \emph{Laplacian eigenvalues} of $G$. Suppose that $L_G$ has eigenvalues $\mu_1,  \mu_2, \ldots , \mu_n$ with the corresponding eigenvectors $\mathbf{x}_1, \mathbf{x}_2,  \ldots , \mathbf{x}_n$ satisfying that $\mathbf{x}_r^{H}\mathbf{x}_r=1$, $\mathbf{x}_r^{H}\mathbf{x}_s=0$ for $r\neq s$, where $\mathbf{x}_r^{H}$ denotes the conjugated transpose of $\mathbf{x}_r$. Define
$$
F_{\mu_r}:=\mathbf{x}_r\mathbf{x}_r^H, ~r=1, 2, \ldots , n,
$$
which is usually called the \emph{eigenprojector} corresponding to  $\mu_r$ of $G$. Note that $\sum_{r=1}^nF_{\mu_r}=I$ (the identity matrix).
Then
\begin{equation}\label{SpecDec11}
L_G=L_G\sum\limits_{r=1}^{n}F_{\mu_r}
=\sum\limits_{r=1}^{n}L_G \mathbf{x}_{r}\mathbf{x}_r^H =\sum\limits_{r=1}^{n}\mu_{r}\mathbf{x}_{r}\mathbf{x}_r^H =\sum\limits_{r=1}^{n}\mu_{r}F_{\mu_r},
\end{equation}
which is called the \emph{spectral decomposition of $L_G$ with respect to the whole eigenvalues} (see \cite[Theorem 2.5.1, Page 27]{Godsil93}). Note that $F_{\mu_r}^{2}=F_{\mu_r}$ and $F_{\mu_r}F_{\mu_s}=\mathbf{0}_n$ for $r\neq s$, where $\mathbf{0}_n$ denotes the zero matrix of size $n$. So, by (\ref{SpecDec11}), we have
\begin{equation}\label{H_ADec2}
U_{L_G}(t)=\sum_{k= 0}^{\infty}\dfrac{(-\mathrm{i})^{k}L_G^{k}t^{k}}{k!}=\sum_{k= 0}^{\infty}\dfrac{(-\mathrm{i})^{k}(\sum_{r=1}^{n}\mu_{r}^{k}F_{\mu_r})t^{k}}{k!} =\sum_{r=1}^{n}\exp(-\mathrm{i}t\mu_{r})F_{\mu_r}.
\end{equation}

Let $\{a,b\}$ be a pair of vertices in $G$. The \emph{Laplacian eigenvalue support} of $\mathbf{e}_{a}-\mathbf{e}_{b}$ in $G$, denoted by $\mathrm{{supp}}_{L_G}(\mathbf{e}_{a}-\mathbf{e}_{b})$, is the set of all eigenvalues $\theta$ of $L_G$ such that
$F_\theta\mathbf(\mathbf{e}_{a}-\mathbf{e}_{b})\neq \mathbf{0}$. Two pairs of vertices $\{a,b\}$ and $\{c,d\}$ are \emph{Laplacian strongly cospectral} if $F_\theta\mathbf(\mathbf{e}_{a}-\mathbf{e}_{b})=\pm F_\theta\mathbf(\mathbf{e}_c-\mathbf{e}_d)$ for each eigenvalue $\theta$ of $L_G$. Let $\Lambda^{+}_{ab,cd}$ and $\Lambda^{-}_{ab,cd}$ denote the sets of all eigenvalues $\theta\in\mathrm{{supp}}_{L_G}(\mathbf{e}_{a}-\mathbf{e}_{b})$ such that $F_\theta\mathbf(\mathbf{e}_{a}-\mathbf{e}_{b})= F_\theta\mathbf(\mathbf{e}_c-\mathbf{e}_d)$
and $F_\theta\mathbf(\mathbf{e}_{a}-\mathbf{e}_{b})=- F_\theta\mathbf(\mathbf{e}_c-\mathbf{e}_d)$, respectively.

The Kronecker's theorem plays an important role in proving the result.

\begin{lemma}\emph{(Kronecker's theorem, see \cite[Chapter 3]{LeZ82})}
\label{Solution}
Let $\theta_{0},\theta_{1},\ldots,\theta_{d}$ and $\zeta_{0},\zeta_{1},\ldots,\zeta_{d}$ be arbitrary real numbers. For the system of inequalities $$|\theta_{k}t-\zeta_{k}|<\epsilon \pmod{2\pi}  ~~~~(k=0,1,\ldots, d)$$
to have consistent real solutions for any arbitrarily small positive number $\epsilon$, it is necessary and sufficient that every time the relation
$$l_{0}\theta_{0}+l_{1}\theta_{1}+\cdots+l_{d}\theta_{d}=0$$
holds, where $l_0, l_1, \ldots , l_d$ are integers, we have the congruence
$$l_{0}\zeta_{0}+l_{1}\zeta_{1}+\cdots+l_{d}\zeta_{d}\equiv0 \pmod {2\pi}.$$
\end{lemma}

\begin{theorem}\label{lp}
Let $G$ be a graph with two pairs of vertices $\{a,b\}$ and $\{c,d\}$. Then $G$ has Pair-LPGST from $\{a,b\}$ to $\{c,d\}$ if and only if the following conditions hold:
\begin{itemize}
\item[\rm (1)] $\mathbf{e}_a-\mathbf{e}_b$ and $\mathbf{e}_c-\mathbf{e}_d$ are Laplacian strongly cospectral.
\item[\rm (2)]  whenever there are integers $\{l_{i}\}$ and $\{m_{j}\}$ such that
\begin{align*}
\sum_{\lambda_{i}\in\Lambda^{+}_{ab,cd}}l_{i}\lambda_{i} +\sum_{\mu_{j}\in\Lambda^{-}_{ab,cd}}m_{j}\mu_{j}&=0,\\
\sum_{i}l_{i}+\sum_{j}m_{j}&=0,
\end{align*}
we have that $\sum\limits_{j}m_{j}$ is even.
\end{itemize}
\end{theorem}

\begin{proof}
We first prove that if $G$ has Pair-LPGST from $\{a,b\}$ to $\{c,d\}$, then Condition (1) holds, that is, Condition (1) is the necessary condition for $G$ to have Pair-LPGST from $\{a,b\}$ to $\{c,d\}$.

Let $\theta_{1},\theta_{2}, \ldots,\theta_{n}$ be the eigenvalues of $L_{G}$ with the corresponding orthonormal eigenvectors $\mathbf{x}_{1}, \mathbf{x}_{2}, \ldots, \mathbf{x}_{n}$.
By (\ref{H_ADec2}), we have
$$
U_{L_G}(t)=\sum_{r=1}^{n}e^{-\mathrm{i}t\theta_{r}}F_{\theta_{r}}.
$$
Note that $G$ has Pair-LPGST from $\{a,b\}$ to $\{c,d\}$. Then for any $\epsilon>0$, there exist a time $\tau$ and a complex scalar $\gamma$ with $|\gamma|=1$ such that
$$
\left|\sum_{r}e^{-\mathrm{i}\tau\theta_{r}}F_{\theta_{r}}
(\mathbf{e}_{a}-\mathbf{e}_{b}) -\gamma(\mathbf{e}_{c}-\mathbf{e}_{d})\right|<\epsilon.
$$
For $\theta_k$, $k=1, 2,\ldots, n$, define
$$
\left|F_{\theta_k}\right|= \max_{|\mathbf{x}|=1} \left|F_{\theta_k}\mathbf{x}\right|.
$$
By \cite[Theorem 5.6.2 (b)]{HornJ2013}, for any vector $\mathbf{x}$, we have
$$
\left|F_{\theta_k}\mathbf{x}\right|\le \left|F_{\theta_k}\right| \left|\mathbf{x}\right|.
$$
By \cite[Equation (5.2.7)]{Meyer2001}, we have $\left|F_{\theta_k}\right|=1$. Thus, for any vector $\mathbf{x}$, we have
$$
\left|F_{\theta_k}\mathbf{x}\right|\le \left|\mathbf{x}\right|.
$$
Then
\begin{align*}
\nonumber \left|e^{-\mathrm{i}\tau\theta_{k}}F_{\theta_k}(\mathbf{e}_{a}-\mathbf{e}_{b})-\gamma F_{\theta_k}(\mathbf{e}_{c}-\mathbf{e}_{d})\right|
&=\left|F_{\theta_k}\left(\sum_{r}e^{-\mathrm{i}\tau\theta_{r}}F_{\theta_r}
(\mathbf{e}_{a}-\mathbf{e}_{b}) -\gamma(\mathbf{e}_{c}-\mathbf{e}_{d})\right)\right|\\
\nonumber &\leq
\left|\sum_{r}e^{-\mathrm{i}\tau\theta_{r}}F_{\theta_r}(\mathbf{e}_{a}-\mathbf{e}_{b}) -\gamma(\mathbf{e}_{c}-\mathbf{e}_{d})\right|
\\
&<\epsilon.
\end{align*}
Note that $\left|e^{-\mathrm{i}\tau\theta_{k}}\right|=\left|\gamma\right|=1$. Then
$$
\left|\left|F_{\theta_k}(\mathbf{e}_{a}-\mathbf{e}_{b})\right|-\left| F_{\theta_k}(\mathbf{e}_{c}-\mathbf{e}_{d})\right|\right| \le \left|e^{-\mathrm{i}\tau\theta_{k}}F_{\theta_k}(\mathbf{e}_{a}-\mathbf{e}_{b})-\gamma F_{\theta_k}(\mathbf{e}_{c}-\mathbf{e}_{d})\right|<\epsilon.
$$
Due to the arbitrariness of $\epsilon$,  we conclude that $$F_{\theta_k}(\mathbf{e}_{a}-\mathbf{e}_{b}) = \pm F_{\theta_k}(\mathbf{e}_{c}-\mathbf{e}_{d}).$$
This proves that Condition (1) holds.

In the following, we prove that for two Laplacian strongly cospectral pairs of vertices $\{a,b\}$ and $\{c,d\}$, $G$ has Pair-LPGST from $\{a,b\}$ to $\{c,d\}$ if and only if Condition (2) holds.

Suppose that $G$ has Pair-LPGST from $\{a,b\}$ to $\{c,d\}$. Then, for any $\epsilon>0$, there is a time $\tau$ such that
\begin{align*}
\left|\frac{1}{2}(\mathbf e_{a}-\mathbf e_{b})^\top e^{-\mathrm{i}\tau L_{G}}(\mathbf e_{c}-\mathbf e_{d})\right|^{2}>1-\epsilon.
\end{align*}
Recall that $F_{\theta_k}=\mathbf{x}_k\mathbf{x}_k^H$, $k=1, 2,\ldots, n$. Let $\mathbf{x}_{k}(u)$ denote the element at the $u$-th position of column vector $\mathbf{x}_{k}$. Then
$$
F_{\theta_k}(\mathbf{e}_{a}-\mathbf{e}_{b}) =(\mathbf{x}_{k}(a)-\mathbf{x}_{k}(b))\mathbf{x}_{k},
$$
and
$$
F_{\theta_k}(\mathbf{e}_{c}-\mathbf{e}_{d}) =(\mathbf{x}_{k}(c)-\mathbf{x}_{k}(d))\mathbf{x}_{k}.
$$
Thus
\begin{align}\label{EquaCloseto1-1}
\nonumber&\left|\frac{1}{2}(\mathbf e_{a}-\mathbf e_{b})^\top e^{-\mathrm{i}\tau L_{G}}(\mathbf e_{c}-\mathbf e_{d})\right|^{2}\\
\nonumber&=\left|\frac{1}{2}\sum_{\lambda_{i}\in\Lambda^{+}_{ab,cd}}e^{-\mathrm{i}\tau\lambda_{i}} (\mathbf e_{a}-\mathbf e_{b})^\top F_{\lambda_i}(\mathbf e_{c}-\mathbf e_{d})+\frac{1}{2}\sum_{\mu_{j}\in\Lambda^{-}_{ab,cd}}{e}^{-\mathrm{i}\tau\mu_{j}} (\mathbf e_{a}-\mathbf e_{b})^\top F_{\mu_j}(\mathbf e_{c}-\mathbf e_{d})\right|^{2}\\
\nonumber&=\left|\frac{1}{2}\sum_{\lambda_{i}\in\Lambda^{+}_{ab,cd}}{e}^{-\mathrm{i}\tau\lambda_{i}} (\mathbf e_{a}-\mathbf e_{b})^\top \mathbf{(x}_{i}(a)-\mathbf{x}_{i}(b))\mathbf{x}_{i}- \frac{1}{2}\sum_{\mu_{j}\in\Lambda^{-}_{ab,cd}}{e}^{-\mathrm{i}\tau\mu_{j}} (\mathbf e_{a}-\mathbf e_{b})^\top (\mathbf{x}_{j}(a)-\mathbf{x}_{j}(b))\mathbf{x}_{j}\right|^{2}\\
\nonumber&=\left|\frac{1}{2}\sum_{\lambda_{i}\in\Lambda^{+}_{ab,cd}}{e}^{-\mathrm{i}\tau \lambda_{i}}(\mathbf{x}_{i}(a)-\mathbf{x}_{i}(b))^{2}
-\frac{1}{2}\sum_{\mu_{j}\in\Lambda^{-}_{ab,cd}}{e}^{-\mathrm{i}\tau \mu_{j}}(\mathbf{x}_{j}(a)-\mathbf{x}_{j}(b))^{2}\right|^{2}\\
& >1-\epsilon.
\end{align}
Note that $\sum_{r=1}^n F_{\theta_r}=I$. Then
$$
\sum_{r=1}^n  \mathbf{x}_{r}(a)^{2}=1  \text{~and~} \sum_{r=1}^n  \mathbf{x}_{r}(a)\mathbf{x}_{r}(b)=0.
$$
Thus
$$
 \sum_{r=1}^n (\mathbf{x}_{r}(a)-\mathbf{x}_{r}(b))^{2}=2.
$$
In order to make Inequality (\ref{EquaCloseto1-1}) correct, we need the phases of ${e}^{-\mathrm{i}\tau\lambda_{i}}$ and $-{e}^{-\mathrm{i}\tau\mu_{j}}$ to be close to lining up to point in the same direction. That is, for any $\epsilon>0$, there is a time $\tau$ such that
\begin{align*}
{e}^{-\mathrm{i}\tau\lambda_{i}}={e}^{-\mathrm{i}(\tau\lambda_{0}+\epsilon)} \text{~and~} {e}^{-\mathrm{i}\tau\mu_{j}}=-{e}^{-\mathrm{i}(\tau\lambda_{0}+\epsilon)},
\end{align*}
for all $i$ and $j$. In other words, for any $\epsilon>0$, there is a time $\tau$ such that
\begin{align*}
&|(\lambda_{i}-\lambda_{0})\tau-0|<\epsilon  \pmod{2\pi},\\[0.15cm]
&|(\mu_{j}-\lambda_{0})\tau-\pi|<\epsilon  \pmod{2\pi}.
\end{align*}
Take $\zeta_i$ to be $0$ for the first set of equations, and $\pi$ for the second set. Lemma \ref{Solution} tells us that this has a solution $\tau$ if and only if, whenever there are integers $l_{i}$ and $m_{j}$ such that
$$\sum_{\lambda_{i}\in\Lambda^{+}_{ab,cd}}l_{i}(\lambda_{i}-\lambda_{0}) +\sum_{\mu_{j}\in\Lambda^{-}_{ab,cd}}m_{j}(\mu_{j}-\lambda_{0})=0,$$
we must have
 $$\sum_{j}m_{j}\pi\equiv0\pmod{2\pi},$$
which implies that $\sum\limits_{j}m_{j}$ is even. This gives Condition (2).

For the converse, assume that Condition (2) holds. In order to prove that $G$ has Pair-LPGST from $\{a,b\}$ to $\{c,d\}$, we need to prove that for any $\epsilon>0$, there exist a time $\tau$ and a complex scalar $\gamma$ with $|\gamma|=1$ such that
$$
\left|\sum_{r}e^{-\mathrm{i}\tau\theta_{r}}F_{\theta_{r}}
(\mathbf{e}_{a}-\mathbf{e}_{b}) -\gamma(\mathbf{e}_{c}-\mathbf{e}_{d})\right|<\epsilon,
$$
that is,
\begin{align}\label{TauEps1-21}
\left|e^{-\mathrm{i}\tau\theta_{k}}F_{\theta_k}(\mathbf{e}_{a}-\mathbf{e}_{b})-\gamma F_{\theta_k}(\mathbf{e}_{c}-\mathbf{e}_{d})\right| &<\epsilon, \text{~~for~} k=1, 2,\ldots, n.
\end{align}
Recall that
$$
F_{\theta_k}(\mathbf{e}_{a}-\mathbf{e}_{b}) = \pm F_{\theta_k}(\mathbf{e}_{c}-\mathbf{e}_{d}), \text{~~for~} k=1, 2,\ldots, n,
$$
and
$$
F_{\theta_k}(\mathbf{e}_{a}-\mathbf{e}_{b}) =(\mathbf{x}_{k}(a)-\mathbf{x}_{k}(b))\mathbf{x}_{k}, \text{~~for~} k=1, 2,\ldots, n.
$$
Plugging them into (\ref{TauEps1-21}), we have
\begin{align*}
\left|e^{-\mathrm{i}\tau\theta_{k}} \mp\gamma \right||\mathbf{x}_{k}(a)-\mathbf{x}_{k}(b)| &<\epsilon, \text{~~for~} k=1, 2,\ldots, n,
\end{align*}
that is,
\begin{align}\label{TauEps1-21-2}
\left|e^{-\mathrm{i}\tau\theta_{k}} \mp\gamma \right|&<\epsilon, \text{~~for~} \theta_{k}\in\mathrm{{supp}}_{L_G}(\mathbf{e}_{a}-\mathbf{e}_{b}).
\end{align}
Here, keep in mind that $|\mathbf{x}_{k}(a)-\mathbf{x}_{k}(b)|=0$ if $\theta_{k}\notin\mathrm{{supp}}_{L_G}(\mathbf{e}_{a}-\mathbf{e}_{b})$.

Set $\gamma=e^{-\mathrm{i}\delta}$. Then (\ref{TauEps1-21-2}) has a solution if and only if
\begin{align} \label{approx}
\delta\approx\theta_{k}\tau-q_{k}\pi,
\end{align}
 where $q_{r}\in\mathbb{Z}$ is even if $\theta_{k}\in\Lambda^{+}_{ab,cd}$, and odd if $\theta_{k}\in\Lambda^{-}_{ab,cd}$. Clearly, a solution to (\ref{approx}) is equivalent to a solution as described in Lemma \ref{Solution} with
 $$t=\tau ~~\text{and}~~\zeta_{k}=\delta+\sigma_{k}\pi,$$
 where $\sigma_{k}=0$ if $\theta_{k}\in\Lambda^{+}_{ab,cd}$, and $\sigma_{k}=1$ if $\theta_{k}\in\Lambda^{-}_{ab,cd}$.

Assume that whenever there are integers $\{l_{i}\}$ and $\{m_{j}\}$ such that
\begin{align*}
\sum_{\lambda_{i}\in\Lambda^{+}_{ab,cd}}l_{i}\lambda_{i} +\sum_{\mu_{j}\in\Lambda^{-}_{ab,cd}}m_{j}\mu_{j}&=0,\\
\sum_{i}l_{i}+\sum_{j}m_{j}&=0,
\end{align*}
we have that $\sum\limits_{j}m_{j}$ is even. Then
 \begin{align*}
 \delta\left(\sum_{i}l_{i}+\sum_{j}m_{j}\right)
 +\sum_{j}m_{j}\pi\equiv0 \pmod {2\pi},
 \end{align*}
which is equivalent to
 $$
 \sum_{i}l_{i}\delta+\sum_{j}m_{j}(\delta+\pi)\equiv0 \pmod {2\pi}.
 $$
By Lemma \ref{Solution}, if Condition (2) holds, then (\ref{TauEps1-21-2}) has a solution, that is, (\ref{TauEps1-21}) is satisfied. Thus, $G$ has Pair-LPGST from $\{a,b\}$ to $\{c,d\}$.

This completes the proof. \qed
\end{proof}

\section{Characterization of Edge-LPGST in paths}
\label{Sec:main3}

In this section, we give a complete characterization of Edge-LPGST in paths. Recall that $P_n$ denotes a path with $n$ vertices, whose vertices from one terminal vertex to the other terminal vertex are labelled by the positive integers from $1$ to $n$ such that vertices with successive labels are adjacent.


\begin{lemma}\emph{(See \cite[Lemma 5.4.4]{QCh18})}
\label{Ei-ValueVector-1}
The Laplacian eigenvalues of $P_{n}$ are $\theta_k=2-2\cos\frac{k\pi}{n}$ with the corresponding eigenvector
$$
\mathbf{x}_k=2\sin\frac{k\pi}{2n}\left(\begin{array}{c}
                                                                 \cos\frac{k\pi}{2n}\\[0.2cm]
                                                                 \cos\frac{3k\pi}{2n}\\[0.2cm]
                                                                 \cos\frac{5k\pi}{2n}\\[0.2cm]
                                                                 \vdots\\[0.2cm]
                                                                 \cos\frac{(2n-3)k\pi}{2n}\\[0.2cm]
                                                                 \cos\frac{(2n-1)k\pi}{2n}\\
                                                               \end{array}
\right),
$$
for $k=0,1,2,\ldots,n-1$.
\end{lemma}

\begin{theorem}\label{FequalS-11}
Let $P_{n}$, $\theta_k$ and $\mathbf{x}_k$ be as in Lemma \ref{Ei-ValueVector-1}, and let $1\leq a \leq n-1$ and $1\leq b \leq n-1$ be two integers. Then $\mathbf{e}_{a}-\mathbf{e}_{a+1}$ and $\mathbf{e}_{n-b}-\mathbf{e}_{n-b+1}$ are Laplacian strongly cospectral if and only if $a=b$. Moreover,
\begin{align*}
  F_{\theta_k}(\mathbf{e}_{a}-\mathbf{e}_{a+1}) &=F_{\theta_k}(\mathbf{e}_{n-a}-\mathbf{e}_{n-a+1})=\mathbf{0}  \text{~~if~and~only~if~$n| ak$},\\[0.15cm]
 F_{\theta_k}(\mathbf{e}_{a}-\mathbf{e}_{a+1}) &=F_{\theta_k}(\mathbf{e}_{n-a}-\mathbf{e}_{n-a+1}) \not=\mathbf{0} \text{~~if~and~only~if~$k$~is~odd~and~$n\nmid ak$},\\[0.15cm]
 F_{\theta_k}(\mathbf{e}_{a}-\mathbf{e}_{a+1}) &=-F_{\theta_k}(\mathbf{e}_{n-a}-\mathbf{e}_{n-a+1}) \not=\mathbf{0} \text{~~if~and~only~if~$k$~is~even~and~$n\nmid ak$}.
\end{align*}
That is,
$$
\mathrm{{supp}}_{L_G}(\mathbf{e}_{a}-\mathbf{e}_{a+1}) =\mathrm{{supp}}_{L_G}(\mathbf{e}_{n-a}-\mathbf{e}_{n-a+1}) =\left\{\theta_k:n\nmid ak\right\}.
$$
\end{theorem}

\begin{proof}
Recall that
$$
F_{\theta_k}(\mathbf{e}_{a}-\mathbf{e}_{a+1}) =(\mathbf{x}_{k}(a)-\mathbf{x}_{k}(a+1))\mathbf{x}_{k},
$$
and
$$
F_{\theta_k}(\mathbf{e}_{n-b}-\mathbf{e}_{n-b+1}) =(\mathbf{x}_{k}(n-b)-\mathbf{x}_{k}(n-b+1))\mathbf{x}_{k},
$$
where $\mathbf{x}_{k}(u)$ denote the element at the $u$-th position of column vector $\mathbf{x}_{k}$. By Lemma \ref{Ei-ValueVector-1}, we have
\begin{align*}
F_{\theta_k}(\mathbf{e}_{a}-\mathbf{e}_{a+1}) &=2\sin\frac{k\pi}{2n} \left(\cos\frac{(2a-1)k\pi}{2n}-\cos\frac{(2a+1)k\pi}{2n}\right)\mathbf{x}_{k}\\
&=4\left(\sin\frac{k\pi}{2n}\right)^2\left(\sin\frac{ak\pi}{n}\right)\mathbf{x}_{k},
\end{align*}
and
\begin{align*}
F_{\theta_k}(\mathbf{e}_{n-b}-\mathbf{e}_{n-b+1}) &=2\sin\frac{k\pi}{2n}\left(\cos\left(k\pi-\frac{(2b+1)k\pi}{2n}\right) -\cos\left(k\pi-\frac{(2b-1)k\pi}{2n}\right) \right)\mathbf{x}_{k}\\
&=4\left(\sin\frac{k\pi}{2n}\right)^2\sin\left(k\pi-\frac{bk\pi}{n}\right)\mathbf{x}_{k}\\
&=(-1)^{k+1}4\left(\sin\frac{k\pi}{2n}\right)^2\left(\sin\frac{bk\pi}{n}\right)\mathbf{x}_{k}.
\end{align*}
Recall that $1\leq a \leq n-1$ and $1\leq b \leq n-1$. Let $k=1$. Then
$$\sin\frac{\pi}{2n}\not=0, ~~ \sin\frac{a\pi}{n}\not=0, \text{~~and~~}\sin\frac{b\pi}{n}\not=0.$$
Thus, in order to make
$$
F_{\theta_1}(\mathbf{e}_{a}-\mathbf{e}_{a+1}) =\pm F_{\theta_1}(\mathbf{e}_{n-b}-\mathbf{e}_{n-b+1}),
$$
we must have $a=b$.

Conversely, if $a=b$, then one can easily verify that $\mathbf{e}_{a}-\mathbf{e}_{a+1}$ and $\mathbf{e}_{n-b}-\mathbf{e}_{n-b+1}$ are Laplacian strongly cospectral. Moreover,
\begin{align*}
  F_{\theta_k}(\mathbf{e}_{a}-\mathbf{e}_{a+1}) &=F_{\theta_k}(\mathbf{e}_{n-a}-\mathbf{e}_{n-a+1})=\mathbf{0}  \text{~~if~and~only~if~$n| ak$},\\[0.15cm]
 F_{\theta_k}(\mathbf{e}_{a}-\mathbf{e}_{a+1}) &=F_{\theta_k}(\mathbf{e}_{n-a}-\mathbf{e}_{n-a+1}) \not=\mathbf{0} \text{~~if~and~only~if~$k$~is~odd~and~$n\nmid ak$},\\[0.15cm]
 F_{\theta_k}(\mathbf{e}_{a}-\mathbf{e}_{a+1}) &=-F_{\theta_k}(\mathbf{e}_{n-a}-\mathbf{e}_{n-a+1}) \not=\mathbf{0} \text{~~if~and~only~if~$k$~is~even~and~$n\nmid ak$}.
\end{align*}
That is,
$$
\mathrm{{supp}}_{L_G}(\mathbf{e}_{a}-\mathbf{e}_{a+1}) =\mathrm{{supp}}_{L_G}(\mathbf{e}_{n-a}-\mathbf{e}_{n-a+1}) =\left\{\theta_k:n\nmid ak\right\}.
$$

This completes the proof. \qed

\end{proof}

\begin{theorem}\label{2nlpgst}
If $n$ is a power of $2$ or an odd prime number, then $P_{n}$ has Edge-LPGST from $\{a,a+1\}$ to $\{n-a,n-a+1\}$, where $1\leq a\leq n-1$ is an integer.
\end{theorem}

\begin{proof}
Let $\zeta_{2n}=e^{\pi \mathrm{i}/n}$. Recall Lemma \ref{Ei-ValueVector-1} that the Laplacian eigenvalues of $P_{n}$ are
$$
\theta_{k}=2-2\cos\frac{\pi k}{n}=2-\left(\zeta_{2n}^{k}+\zeta_{2n}^{2n-k}\right),~~k=0,1,2,\ldots,n-1.
$$
By Theorem \ref{FequalS-11},
$$
S:=\mathrm{{supp}}_{L_G}(\mathbf{e}_{a}-\mathbf{e}_{a+1}) =\mathrm{{supp}}_{L_G}(\mathbf{e}_{n-a}-\mathbf{e}_{n-a+1}) =\{\theta_k:n\nmid ak\},
$$
and
\begin{align}\label{InSupp-11}
&\theta_k \notin S \text{~if~and~only~if~}\theta_{n-k} \notin S, \text{~~for~}k=1,2,\ldots,n-1.
\end{align}
Let
\begin{align}\label{Def-sigma-11}
\sigma_{k}=\left\{\begin{array}{ll}
                         1, & \text{if~} F_{\theta_k}(\mathbf{e}_{a}-\mathbf{e}_{a+1}) =-F_{\theta_k}(\mathbf{e}_{n-a}-\mathbf{e}_{n-a+1})\not=0\\[0.25cm]
                         0, & \text{otherwise}.
                       \end{array}\right.
\end{align}
Note that $\theta_0\notin S$. In order to prove that $P_{n}$ has Edge-LPGST from $\{a,a+1\}$ to $\{n-a,n-a+1\}$, by Theorem \ref{lp}, we verify a more general condition, that is, whenever there are integers $l_{1},l_{2},\ldots,l_{n-1}$ such that
 \begin{align}
\label{Po2Prime--eq1} \sum_{k=1}^{n-1}l_{k}\theta_{k}& =\sum_{k=1}^{n-1}l_{k}\left(2-\left(\zeta_{2n}^{k}+\zeta_{2n}^{2n-k}\right)\right)=0, \\
\label{Po2Prime--eq2} \sum_{k=1}^{n-1}l_{k}&=0.
\end{align}
we must have that $\sum\limits_{i=1}^{n-1} \sigma_{i} l_{i}$ is even.

Assume that there are integers $l_{1},l_{2},\ldots,l_{n-1}$ such that (\ref{Po2Prime--eq1}) and (\ref{Po2Prime--eq2}) are satisfied. Define
$$
L(x)=\sum_{k=1}^{n-1}l_{k}x^{k}+\sum_{k=n+1}^{2n-1}l_{2n-k}x^{k}.
$$
By (\ref{Po2Prime--eq1}) and (\ref{Po2Prime--eq2}), $\zeta_{2n}$ is a root of $L(x)$. Let $\Phi_{2n}(x)$ denote the minimal polynomial of $\zeta_{2n}$. Then $\Phi_{2n}(x)$ divides $L(x)$.

\noindent\emph{Case 1.}
If $n$ is power of $2$, then $\Phi_{2n}(x)=1+x^{n}$. Performing long division of  polynomials, the quotient of the division of $L(x)$ by $\Phi_{2n}(x)$ is
$$\sum_{k=1}^{n-1}l_{k}x^{n-k},$$
and the remainder is
$$
(l_{n-1}-l_{1})x^{n-1}+(l_{n-2}-l_{2})x^{n-2}+\cdots+(l_{1}-l_{n-1})x.
$$
Thus $\Phi_{2n}(x)$ divides $L(x)$ if and only if $l_{k}=l_{n-k}$, for all $k=1,2,\ldots,n-1$. By (\ref{InSupp-11}) and Theorem \ref{FequalS-11}, we must have that $\sum\limits_{i=1}^{n-1} \sigma_{i} l_{i}$ is even. Thus, $P_{n}$ has Edge-LPGST from $\{a,a+1\}$ to $\{n-a,n-a+1\}$ in this case.

\noindent\emph{Case 2.}
If $n$ is odd prime, then $\Phi_{2n}(x)=1-x+x^{2}-\cdots-x^{n-2}+x^{n-1}$. Performing long division of polynomials, the quotient of the division of $L(x)$ by $\Phi_{2n}(x)$ is
$$(l_{n-1}-l_{1})+l_{n-1}x+\sum_{k=2}^{n-1}(l_{k}+l_{k-1})x^{n-k+1}+l_{1}x^{n},$$
and the remainder is
 \begin{align*}
 &((l_{n-2}-l_{2})+(l_{n-1}-l_{1}))x^{n-2}+((l_{n-3}-l_{3})-(l_{n-1}-l_{1}))x^{n-3}+\cdots+\\
 &((l_{2}-l_{n-2})-(l_{n-1}-l_{1}))x^{2}+((l_{1}-l_{n-1})+(l_{n-1}-l_{1}))x-(l_{n-1}-l_{1}).
 \end{align*}
Thus $\Phi_{2n}(x)$ divides $L(x)$ if and only if $l_{k}=l_{n-k}$, for all $k=1,2,\ldots,n-1$. Then $\sum\limits_{i \text{~is~even}}l_{i}=\sum\limits_{i\text{~is~odd}}l_{i}$. Together with (\ref{Po2Prime--eq2}), we have $\sum\limits_{i \text{~is~even}}l_{i}=\sum\limits_{i\text{~is~odd}}l_{i}=0$. By (\ref{InSupp-11}) and Theorem \ref{FequalS-11}, we must have that $\sum\limits_{i=1}^{n-1} \sigma_{i} l_{i}=0$ is even. Thus, $P_{n}$ has Edge-LPGST from $\{a,a+1\}$ to $\{n-a,n-a+1\}$ in this case.

This completes the proof. \qed
\end{proof}

Next, we study the existence of Edge-LPGST in $P_{n}$ from $\{a,a+1\}$ to $\{n-a,n-a+1\}$, where $n$ is neither a power of $2$ nor an odd prime number and $1\leq a\leq n-1$ is an integer.

The following result plays an important role in the sequel.

\begin{lemma}\emph{(See \cite{Bo19}, Lemma 4.3.2)}\label{composite}
Let $n=km$, where $k$ is a positive integer and $m>1$ is an odd integer. If $0\leq c<k$ is an integer, then
$$\sum_{j=0}^{m-1}(-1)^{j}\cos\frac{(c+jk)\pi}{n}=0.$$
\end{lemma}


\begin{theorem}\label{nopgst1}
Let $n=2^{t}r$, where $t$ is a nonnegative integer and $r$ is an odd composite number. Then $P_{n}$ has no Edge-LPGST from $\{a,a+1\}$ to $\{n-a,n-a+1\}$, where $1\leq a\leq n-1$ is an integer.
\end{theorem}

\begin{proof}
\noindent\emph{Case 1.} $t>0$. We consider the following  cases.

\noindent\emph{Case 1.1.} $r$ has no prime factor $p$ such that $p\left|~\frac{2^{t}r}{\gcd(a,2^{t}r)}\right.$.
In this case, $r|\gcd(a,2^{t}r)$, and then $r|a$. Note that $\{a,a+1\}$ and $\{n-a,n-a+1\}$ are two distinct pairs of vertices. Then $a\neq2^{t-1}r$, that is, $r\le a\leq(2^{t-1}-1)r$. Thus, we have $t\geq2$ and $4\left|~\frac{2^{t}r}{\gcd(a,2^{t}r)}\right.$. By Theorem \ref{FequalS-11}, if $k$ is not a multiple of $4$, then $\theta_{k}\in S:=\mathrm{{supp}}_{L_G}(\mathbf{e}_{a}-\mathbf{e}_{a+1}) $. For $1\le k \le n-1$, define
 \begin{equation*}
	\label{relation}
l_{k}=\left\{
           \begin{array}{rl}
    1,   &\text{if~}k\equiv1, 2^{t}+2 \pmod {2^{t+1}},\\[0.2cm]
   -1,   &\text{if~}k\equiv2, 2^{t}+1 \pmod {2^{t+1}},\\[0.2cm]
   0,   & \text{otherwise.}
  \end{array}\right.
\end{equation*}
Then
$$
\sum_{k=1}^{n-1}l_{k}=0.
$$
For $c\in\{1,2\}$, by Lemma \ref{composite}, we have
$$
\sum_{j=0}^{r-1}(-1)^{j}\theta_{c+j2^{t}} =\sum_{j=0}^{r-1}(-1)^{j}\left(2-2\cos\frac{(c+j2^{t})\pi}{n}\right)=2.
$$
Then
$$
 \sum_{k=1}^{n-1}l_{k}\theta_{k} =\sum_{j=0}^{r-1}(-1)^{j}\theta_{1+j2^{t}}
-\sum_{j=0}^{r-1}(-1)^{j}\theta_{2+j2^{t}}=0.$$
On the other hand, by Theorem \ref{FequalS-11}, if $k\equiv2, 2^{t}+2 \pmod {2^{t+1}}$, then
$$
F_{\theta_k}(\mathbf{e}_{a}-\mathbf{e}_{a+1}) =-F_{\theta_k}(\mathbf{e}_{n-a}-\mathbf{e}_{n-a+1}) \not=\mathbf{0},
$$
and if $k\equiv1, 2^{t}+1 \pmod {2^{t+1}}$, then
$$
F_{\theta_k}(\mathbf{e}_{a}-\mathbf{e}_{a+1}) =F_{\theta_k}(\mathbf{e}_{n-a}-\mathbf{e}_{n-a+1}) \not=\mathbf{0}.
$$
Thus,
$$
\sum\limits_{k=1}^{n-1}\sigma_kl_{k}=-1
$$
is odd, where $\sigma_k$ is as defined in (\ref{Def-sigma-11}). Therefore, by Theorem \ref{lp}, $P_{n}$ has no Edge-LPGST from $\{a,a+1\}$ to $\{n-a,n-a+1\}$ in this case.

\noindent\emph{Case 1.2.} $r$ has a prime factor $p$ such that $p\left|~\frac{2^{t}r}{\gcd(a,2^{t}r)}\right.$. By Theorem \ref{FequalS-11}, if $k$ is not a multiple of $p$, then $\theta_{k}\in S$.  For $1\le k \le n-1$, define
 \begin{equation*}
	\label{relation}
l_{k}'=\left\{
           \begin{array}{rl}
    1,   & \text{if~}k\equiv1,2^{t}p+2 \pmod {2^{t+1}p},\\[0.2cm]
   -1,   &\text{if~}k\equiv2,2^{t}p+1 \pmod {2^{t+1}p},\\[0.2cm]
   0,   & \text{otherwise.}
  \end{array}\right.
\end{equation*}
Then
$$
\sum_{k=1}^{n-1}l_{k}'=0.
$$
For $c\in\{1,2\}$, by Lemma \ref{composite}, we have
$$
\sum_{j=0}^{r/p-1}(-1)^{j}\theta_{c+j2^{t}p} =\sum_{j=0}^{r/p-1}(-1)^{j}\left(2-2\cos\frac{(c+j2^{t}p)\pi}{n}\right)=2.
$$
Thus,
$$\sum_{k=1}^{n-1}l_{k}'\theta_{k}=\sum_{j=0}^{r/p-1}(-1)^{j}\theta_{1+j2^{t}p}
-\sum_{j=0}^{r/p-1}(-1)^{j}\theta_{2+j2^{t}p}=0.$$
On the other hand, by Theorem \ref{FequalS-11}, if $k\equiv2, 2^{t}p+2 \pmod {2^{t+1}p}$, then
$$
F_{\theta_k}(\mathbf{e}_{a}-\mathbf{e}_{a+1}) =-F_{\theta_k}(\mathbf{e}_{n-a}-\mathbf{e}_{n-a+1}) \not=\mathbf{0},
$$
and if $k\equiv1, 2^{t}p+1 \pmod {2^{t+1}p}$, then
$$
F_{\theta_k}(\mathbf{e}_{a}-\mathbf{e}_{a+1}) =F_{\theta_k}(\mathbf{e}_{n-a}-\mathbf{e}_{n-a+1}) \not=\mathbf{0}.
$$
Thus,
$$
\sum\limits_{k=1}^{n-1}\sigma_kl_{k}'=-1
$$
is odd, where $\sigma_k$ is as defined in (\ref{Def-sigma-11}). Therefore, by Theorem \ref{lp}, $P_{n}$ has no Edge-LPGST from $\{a,a+1\}$ to $\{n-a,n-a+1\}$ in this case.

\noindent\emph{Case 2.} $t=0$.
In this case, $r$ has a prime factor $p$ such that $p\left|\frac{r}{\gcd(a,r)}\right.$. Otherwise, we have $r|a$, contradicting that $a<r=n$. By Theorem \ref{FequalS-11}, if $k$ is not a multiple of $p$, then $\theta_{k}\in S$. For $1\le k \le n-1$, define
 \begin{equation*}
	\label{relation}
l_{k}''=\left\{
           \begin{array}{rl}
    1,   & \text{if~}k\equiv1,p+2 \pmod {2p},\\[0.2cm]
   -1,   &\text{if~}k\equiv2,p+1 \pmod {2p},\\[0.2cm]
   0,    &\text{otherwise.}
   \end{array}\right.
\end{equation*}
Similar to Case 1.2, one can easily verify that
$$
\sum_{k=1}^{n-1}l_{k}''=0 \text{~~and~~} \sum_{k=1}^{n-1}l_{k}''\theta_{k} =0.$$
But,
$$
\sum\limits_{k=1}^{n-1}\sigma_kl_{k}''=-1
$$
is odd, where $\sigma_k$ is as defined in (\ref{Def-sigma-11}). Therefore, by Theorem \ref{lp}, $P_{n}$ has no Edge-LPGST from $\{a,a+1\}$ to $\{n-a,n-a+1\}$ in this case.
\qed
\end{proof}

\begin{theorem}
Let $n=2^{t}p$, where $t$ is a positive integer and $p$ is an odd prime. Then $P_{n}$ has Edge-LPGST from $\{a,a+1\}$ to $\{n-a,n-a+1\}$ if and only if $1\leq a\leq n-1$ is a multiple of $2^{t-1}$.
\end{theorem}

\begin{proof}
We use the method of proving Theorem \ref{2nlpgst} to complete the proof. Let $\zeta_{2n}=e^{\pi \mathrm{i}/n}$. Recall Lemma \ref{Ei-ValueVector-1} that the Laplacian eigenvalues of $P_{n}$ are
$$
\theta_{k}=2-2\cos\frac{\pi k}{n}=2-\left(\zeta_{2n}^{k}+\zeta_{2n}^{2n-k}\right),~~k=0,1,2,\ldots,n-1.
$$
Assume that there are integers $l_{1},l_{2},\ldots,l_{n-1}$ such that \begin{align}
\label{Th3.6--Po2Prime--eq1} \sum_{k=1}^{n-1}l_{k}\theta_{k}& =\sum_{k=1}^{n-1}l_{k}\left(2-\left(\zeta_{2n}^{k}+\zeta_{2n}^{2n-k}\right)\right)=0, \\
\label{Th3.6--Po2Prime--eq2} \sum_{k=1}^{n-1}l_{k}&=0.
\end{align}
Define
$$
L(x)=\sum_{k=1}^{n-1}l_{k}x^{k}+\sum_{k=n+1}^{2n-1}l_{2n-k}x^{k}.
$$
By (\ref{Th3.6--Po2Prime--eq1}) and (\ref{Th3.6--Po2Prime--eq2}), $\zeta_{2n}$ is a root of $L(x)$. Let $\Phi_{2n}(x)$ denote the minimal polynomial of $\zeta_{2n}$. Then
$$
\Phi_{2n}(x)=\sum_{i=0}^{p-1}(-1)^{i}x^{2^{t}i},
$$
and $\Phi_{2n}(x)$ divides $L(x)$. Performing long division of  polynomials, the quotient of the division of $L(x)$ by $\Phi_{2n}(x)$ is
$$\sum_{i=1}^{2^{t}}l_{i}x^{2^{t}(p+1)-i} +\sum_{i=2^{t}+1}^{2^{t}p-1}(l_{i}+l_{i-2^t})x^{2^{t}(p+1)-i}
+l_{2^{t}(p-1)}x^{2^t} +\sum_{i=1}^{2^{t}-1}(l_{2^{t}p-i}+l_{2^{t}(p-1)+i}-l_i)x^{2^{t}-i} +(l_{2^{t}(p-1)}-l_{2^t}),$$
and the remainder is
\begin{align*}
&\sum_{i=1}^{2^{t}}((l_{2^{t}(p-1)-i}-l_{2^{t}+i})+(l_{2^{t}p-i}-l_{i}))x^{2^{t}(p-1)-i}
+\sum_{i=1}^{2^{t}}((l_{2^{t}(p-2)-i}-l_{2^{t}\times2+i})-(l_{2^{t}p-i}-l_{i}))x^{2^{t}(p-2)-i}\\
&+\cdots+\sum_{i=1}^{2^{t}}((l_{2^{t}\times2-i}-l_{2^{t}(p-2)+i})+(l_{2^{t}p-i}-l_{i}))x^{2^{t}\times2-i}\\
&+\sum_{i=1}^{2^{t}-1}((l_{2^{t}-i}-l_{2^{t}(p-1)+i})-(l_{2^{t}p-i}-l_{i}))x^{2^{t}-i} -(l_{2^{t}(p-1)}-l_{2^{t}}).
\end{align*}
Thus $\Phi_{2n}(x)$ divides $L(x)$ if and only if  for $i=1,2,\ldots,2^{t}-1$ and $j=1,2, \ldots,p-1$,
\begin{align}
\label{Lequal-1}l_{2^{t}(p-j)- i}-l_{2^{t}j+ i}&=(-1)^{j}(l_{2^{t}p-i}-l_{i}),\\[0.2cm]
\label{Lequal-2} l_{2^{t}(p-j)}-l_{2^{t}j}&=0.
\end{align}
\noindent\emph{Case 1.} $a$ is a multiple of $2^{t-1}$. Let $S:=\mathrm{{supp}}_{L_G}(\mathbf{e}_{a}-\mathbf{e}_{a+1})$. By Theorem \ref{FequalS-11}, if $2p| k$, then $\theta_{k}\notin S$. In the following, we set $l_{k}=0$ if $\theta_{k}\notin S$.

\noindent\emph{Case 1.1.} $t\ge2$. Note that $2^{t}\neq0\pmod {2p}$. Then for any $i\in\{2,4,\ldots,2^{t}-2\}$ and $2p \nmid i$, there is $j_i\in \{1,2,\ldots,p-1\}$ such that $2^{t}j_i\equiv -i \pmod {2p}$, that is, $2p\mid(2^{t}j_i+i)$. Thus, we have $\theta_{2^{t}j_i+i}\notin S$ and $l_{2^{t}j_i+i}=0$ for any $i\in\{2,4,\ldots,2^{t}-2\}$ and $2p \nmid i$. Recall (\ref{InSupp-11}) that
$$
\theta_k \notin S \text{~if~and~only~if~}\theta_{n-k} \notin S, \text{~~for~}k=1,2,\ldots,n-1.
$$
Then we have $l_{2^{t}(p-j_i)-i}=l_{2^{t}j_i+i}=0$ for any $i\in\{2,4,\ldots,2^{t}-2\}$ and $2p \nmid i$. By (\ref{Lequal-1}),  $l_{i}=l_{n-i}$ for any $i\in\{2,4,\ldots,2^{t}-2\}$ and $2p \nmid i$. Notice also that $l_{i}=l_{n-i}$ if $i\in\{2,4,\ldots,2^{t}-2\}$ and $2p | i$. Thus, $l_{i}=l_{n-i}$ for any $i\in\{2,4,\ldots,2^{t}-2\}$. This together with (\ref{Lequal-1}) and (\ref{Lequal-2}) implies that $l_{i}=l_{n-i}$ for any even $i\in\{2,4,\ldots,n-2\}$. By Theorem \ref{FequalS-11}, we must have that $\sum\limits_{k=1}^{n-1} \sigma_{k} l_{k}$ is even, where $\sigma_k$ is as defined in (\ref{Def-sigma-11}).  Therefore, by Theorem \ref{lp}, $P_{n}$ has Edge-LPGST from $\{a,a+1\}$ to $\{n-a,n-a+1\}$ in this case.

\noindent\emph{Case 1.2.} $t=1$. By (\ref{Lequal-2}), we have $l_{2(p-j)}-l_{2j}=0$ for $j=1,2, \ldots,p-1$, that is, $l_{i}=l_{n-i}$ for any even $i\in\{2,4,\ldots,n-2\}$. Similar to Case 1.1, by Theorems \ref{FequalS-11} and \ref{lp}, $P_{n}$ has Edge-LPGST from $\{a,a+1\}$ to $\{n-a,n-a+1\}$ in this case.

\noindent\emph{Case~2.}
$a$ is not a multiple of $2^{t-1}$. In this case, $t\ge2$. Otherwise, if $t=1$, then $2^{t-1}=1$. This implies that any $a$ is a multiple of $2^{t-1}$, a contradiction to the assumption.  So, we have $4\left|~\frac{2^{t}p}{\gcd(a,2^{t}p)}\right.$. By Theorem \ref{FequalS-11}, if $k$ is not a multiple of $4$, then $\theta_{k}\in S:=\mathrm{{supp}}_{L_G}(\mathbf{e}_{a}-\mathbf{e}_{a+1}) $. By a slight refinement of the proof of Case 1.1 in Theorem \ref{nopgst1}, one can easily verify that $P_{n}$ has no Edge-LPGST between $\{a,a+1\}$ and $\{n-a,n-a+1\}$.

This completes the proof. \qed

\end{proof}

\end{document}